\documentclass[12pt, reqno]{amsart}
\usepackage{amsmath, amsthm, amscd, amsfonts, amssymb, graphicx, color}
\usepackage[bookmarksnumbered, colorlinks, plainpages]{hyperref}

\newtheorem{theorem}{Theorem}
\theoremstyle{plain}

\newtheorem{definition}{Definition}
\newtheorem{example}{Example}

\newtheorem{lemma}{Lemma}

\numberwithin{equation}{section}

\begin{document}

\title[Hermite-Hadamard type inequalities]{The Hermite-Hadamard type inequalities
for operator $s$-convex functions}
\author[Ghazanfari]{ Amir G. Ghazanfari}
\address{Department of Mathematics\\
Lorestan University, P.O. Box 465\\
Khoramabad, Iran.}

\email{ghazanfari.amir@gmail.com}
\maketitle
\begin{abstract}
In this paper we introduce operator $s$-convex
functions and establish some Hermite-Hadamard type inequalities in which some operator $s$-convex
functions of positive operators in Hilbert spaces are involved.
\end{abstract}
\section{introduction}

The following inequality holds for any convex function $f$ defined on $\mathbb{R}$ and $a,b \in \mathbb{R}$, with $a<b$
\begin{equation}\label{1.1}
f\left(\frac{a+b}{2}\right)\leq \frac{1}{b-a}\int_a^b f(x)dx\leq \frac{f(a)+f(b)}{2}
\end{equation}
Both inequalities hold in the reversed direction if $f$ is concave. The inequality (\ref{1.1}) is
known in the literature as the Hermite-Hadamard's inequality. We note that
the Hermite-Hadamard's inequality may be regarded as a refinement of the concept of convexity and it follows easily from
Jensen's inequality.
The classical Hermite-Hadamard inequality provides estimates of the mean value
of a continuous convex function $f : [a, b] \rightarrow \mathbb{R}$.

In the paper \cite{hud} Hudzik and Maligranda considered, among others, two classes of functions which are $s$-convex in
the first and second senses. These classes are defined in the following way: a function $f:[0,\infty)\to \mathbb{R}$ is
said to be \emph{$s$-convex in the first sense} if
\begin{equation*}
f(\alpha x+\beta y)\leq\alpha^{s}f(x)+\beta^{s}f(y)
\end{equation*}
holds for all $x,y\in [0,\infty)$ and $\alpha,\beta\geq0$ with $\alpha^{s}+\beta^{s}=1$. The class of $s$-convex functions in
the first sense is usually denoted with $K_{s}^{1}$.

 A function $f:\mathbb{R}^+\rightarrow \mathbb{R}$ where $\mathbb{R}^+=[0,+\infty)$, is said to be $s$-convex in the second
sense if
 \begin{equation*}
f(\lambda x+(1-\lambda)y)\leq\lambda^{s}f(x)+(1-\lambda)^{s}f(y)
\end{equation*}
holds for all $x,y\in [0,\infty),\lambda\in [0,1]$ and for some fixed $s\in (0,1]$. The class of $s$-convex functions in the second
sense is usually denoted with $K_{s}^{2}$.
It can be easily seen that for $s=1$, $s$-convexity reduces to ordinary convexity of functions defined on $[0,\infty)$.\\
 It is proved in \cite{hud} that if $s\in (0,1)$ then $f\in K_{s}^{2}$ implies $f([0,\infty))\subseteq[0,\infty)$, i.e., they proved
 that all functions from $K_{s}^{2},s\in (0,1)$, are nonnegative. The following example can be found in \cite{hud}.

\begin{example} \label{e1}
Let $s\in (0,1)$ and $a,b,c\in \mathbb{R}$. We define function $f:[0,\infty)\to \mathbb{R}$ as \\
\[
        f(t)=\left\{\begin{array}{rcl}
        a, $\qquad\qquad$  t=0,\\
        bt^{s}+c, $\qquad $ t>0.\\
        \end{array}\right.
\]
It can be easily checked that
\begin{enumerate}
\item[(i)]
If $b\geq0$ and $0\leq c\leq a$, then $f\in K_{s}^{2}$,
\item[(ii)]
If $b>0$ and $c<0$, then $f\not \in K_{s}^{2}$.
\end{enumerate}
\end{example}

In Theorem 4 of \cite{hud} both definitions of the $s$-convexity have been compared as follows:\\
(i) Let $0<s\leq1$. If $f\in K^2_s$ and $f(0)=0$, then $f\in K^1_s$,\\
(ii) Let $0<s_1<s_2\leq1$. If $f\in K^2_{s_2}$ and $f(0)=0$, then $f\in K^2_{s_1}$,\\
(iii) Let $0<s_1<s_2\leq1$. If $f\in K^1_{s_2}$ and $f(0)\leq0$, then $f\in K^1_{s_1}$.

In \cite{dra1}, Dragomir and Fitzpatrick proved the following variant of Hadamard's
inequality which holds for $s$-convex functions in the second sense:

\begin{theorem}\label{t1}
Suppose that $f:[0,\infty)\to [0,\infty)$ is an $s$-convex function in the second sense, where $s\in (0,1)$ and let
$a,b\in [0,\infty)$, $a< b$. If $f\in L^{1}[a,b]$, then the following inequalities hold:
\begin{equation}\label{1.2}
2^{s-1}f\left(\frac{a+b}{2}\right)\leq\frac{1}{b-a}\int_{a}^{b}f(x)dx\leq\frac{f(a)+f(b)}{s+1}
\end{equation}
the constant $k=\frac{1}{s+1}$ is the best possible in the second inequality in (\ref{1.2}).
The above inequalities are sharp.
\end{theorem}

The Hermite-Hadamard inequality has several applications in nonlinear analysis and
the geometry of Banach spaces, see \cite{kik}.
In recent years several extensions and generalizations have been considered for
classical convexity. We would like to refer the reader to \cite{bar, dra3, wu} and references
therein for more information. A number of papers have been written on this inequality providing
some inequalities analogous to Hadamard's inequality
given in (\ref{1.1}) involving two convex functions, see \cite{pach, bak, tun}.
Pachpatte in \cite{pach} has proved the following theorem for the product of two convex functions.
\begin{theorem}\label{t2}
Let $f$ and $g$ be real-valued, nonnegative and convex functions on $[a, b]$. Then
\begin{equation*}
\frac{1}{b-a}\int_{a}^{b}f(x)g(x)dx\leq\frac{1}{3}M(a, b)+\frac{1}{6}N(a, b),
\end{equation*}
\begin{equation*}
2f\left(\frac{a+b}{2}\right)g\left(\frac{a+b}{2}\right)\leq\frac{1}{b-a}\int_{a}^{b}f(x)g(x)dx+\frac{1}{6}M(a, b)+\frac{1}{3}N(a, b),
\end{equation*}
where $M(a, b)=f(a)g(a)+f(b)g(b), N(a, b)=f(a)g(b)+f(b)g(a)$.
\end{theorem}

Kirmaci et al. in \cite {kirm} have proved the following theorem for the product of two $s$-convex functions, which is a generalization of Theorem \ref{t2}.
\begin{theorem}\label{t3}
Let $f,g:[0,\infty)\to [0,\infty)$ be $s_1$-convex and $s_2$-convex functions in the second sense respectively, where $s_1,s_2\in (0,1)$. Let $a,b\in [0,\infty)$, $a< b$. If $f,g\text{ and }fg\in L^{1}([a,b])$ then
\begin{equation*}
\frac{1}{b-a}\int_{0}^{1}f(x)g(x)dx
\leq\frac{1}{s_1+s_2+1}M(a,b)+\beta(s_1+1,s_2+1)N(a,b).
\end{equation*}
where $M(a,b)=f(a)g(a)+f(b)g(b), N(a,b)=f(a)g(b)+f(b)g(a)$.
\end{theorem}

In this paper we show that Theorem \ref{t1} and Theorem \ref{t3} hold for operator s-convex functions in a convex subset $K$ of $B(H)^+$
the set of positive operators in $B(H)$. We also obtain some integral inequalities for the product of two operator $s$-convex functions.

\section{operator $s$-convex functions}

First, we review the operator order in $B(H)$ and the continuous functional calculus for a bounded selfadjoint operator.
For selfadjoint operators $A, B \in B(H)$ we write $A\leq B ( \text{or}~ B\geq A)$ if $\langle Ax,x\rangle\leq\langle Bx,x\rangle$ for every vector $x\in H$,
we call it the operator order.

Now, let $A$ be a bounded selfadjoint linear operator on a complex Hilbert space $(H; \langle .,.\rangle)$
and $C(Sp(A))$ the $C^*$-algebra of all continuous complex-valued functions on the spectrum of $A$.
The Gelfand map establishes a $*$-isometrically isomorphism $\Phi$ between
$C(Sp(A))$ and the $C^*$-algebra $C^*(A)$ generated by $A$  and the identity operator $1_H$ on $H$ as
follows (see for instance \cite[p.3]{fur}):
For $f, g\in C (Sp (A))$ and $\alpha,\beta\in\mathbb{C}$
\begin{align*}
&(i)~\Phi(\alpha f+\beta g)=\alpha\Phi(f)+\beta\Phi(g);\\
&(ii)~\Phi(fg)=\Phi(f)\Phi(g)~\text{and}~\Phi(f^*)=\Phi(f)^*;\\
&(iii)~\|\Phi(f)\|=\|f\|:=\sup_{t\in Sp(A)} |f(t)|;\\
&(iv)~\Phi(f_0)=1 ~\text{and}~ \Phi(f_1)=A,~\text{ where } f_0(t)=1~ \text{and } f_1(t)=t, \text{for}~\\
&t\in Sp(A).
\end{align*}
If $f$ is a continuous complex-valued functions on $Sp(A)$, the element $\Phi(f)$ of $C^*(A)$ is denoted by
$f(A)$, and we call it the continuous functional calculus for a bounded selfadjoint operator $A$.

If $A$ is a bounded selfadjoint operator and $f$ is a real-valued continuous function on $Sp (A)$,
then $f (t) \geq 0$ for any $t \in Sp (A)$ implies that $f (A) \geq 0$, i.e., $f (A)$ is a positive
operator on $H$. Moreover, if both $f$ and $g$ are real-valued functions on $Sp (A)$ such that 
$f(t) \leq g (t)$ for any $t\in sp(A)$, then $ f(A)\leq f(B)$ in the operator order in $B(H)$.

A real valued continuous function $f$ on an interval $I$ is said to be operator convex
(operator concave) if
\begin{equation*}
f((1 -\lambda)A +\lambda B)\leq (\geq) (1-\lambda) f (A) + \lambda f (B)
\end{equation*}
in the operator order in $B(H)$, for all $ \lambda\in [0, 1] $ and for every bounded self-adjoint operators $A$ and $B$
in $B(H)$ whose spectra are contained in $I$.

As examples of such functions, we give the following examples, another proof of them and further examples can be found in \cite{fur}.
\begin{example}
\begin{enumerate}
\item[(i)] The convex function $ f(t)=\alpha t^{2}+\beta t+\gamma ~~(\alpha\geq0, \beta,\gamma\in \mathbb{R})$ is operator convex on every interval. To see it, for all self-adjoint
operators $A$ and $B$:
\begin{multline*}
\frac{f(A)+f(B)}{2}-f\left(\frac{A+B}{2}\right)=\alpha\left(\frac{A^2+B^2}{2}-\left(\frac{A+B}{2}\right)^2\right)\\
+\beta\left(\frac{A+B}{2}-\frac{A+B}{2}\right)+(\gamma-\gamma)\\
=\frac{\alpha}{4}(A^2+B^2-AB-BA)=\frac{\alpha}{4}(A-B)^2\geq0.
\end{multline*}
\item[(ii)] The convex function $ f(t)=t^{3} $ on $ [0,\infty) $ is not operator convex. In fact, if we put
\begin{equation*}
A=\left[ \begin{array}{cc}
3&-1\\
-1&1
\end{array}\right] \qquad\&\qquad
B=\left[ \begin{array}{cc}
1&0\\
0&0
\end{array}\right],
\end{equation*}
then
\begin{equation*}
\frac{A^3+B^3}{2}-\left(\frac{A+B}{2}\right)^3=\frac{1}{8}\left[ \begin{array}{cc}
67&-34\\
-34&17
\end{array}\right] \ngeq 0.
\end{equation*}
\end{enumerate}
\end{example}
For some fundamental results on operator convex (operator concave) and operator monotone functions, see \cite{fur} and
the references therein.
%-------------------------------------------------------------------------------------------------------------------------------------------------------%

We denoted by $B(H)^+$ the set of all positive operators in $B(H)$ and
\[
C(H):=\{ A\in B(H)^+:~AB+BA\geq0, ~\text{ for all } B\in B(H)^+\}.
\]
It is obvious that $C(H)$ is a closed convex cone in $B(H)$.

\begin{definition}
Let $I$ be an interval in $[0,\infty)$ and $K$ be a convex subset of $B(H)^+$. A continuous function $f:I\rightarrow \mathbb{R}$ is said to be operator s-convex on $I$ for operators in $K$ if
\begin{align*}
f ((1 -\lambda)A +\lambda B)\leq (1-\lambda)^s f (A) + \lambda^s f (B)
\end{align*}
in the operator order in $B(H)$, for all $ \lambda\in [0, 1] $ and for every positive operators $A$ and $B$
in $K$ whose spectra are contained in $I$ and for some fixed $s\in (0,1]$.
For $K=B(H)^+$ we say $f$ is operator s-convex on $I$.
\end{definition}
First of all we state the following lemma.
\begin{lemma}
If $f$ is operator $s$-convex on $[0,\infty)$ for operators in $K$, then $f(A)$ is positive for every $A\in K$.
\end{lemma}
\begin{proof}
For $A\in K$, we have
\[
f(A)=f\left(\frac{A }{2}+\frac{A}{2}\right)\leq \left(\frac{1}{2}\right)^s f (A) + \left(\frac{1}{2}\right)^s f (A)= 2^{1-s}f(A).
\]
This implies that $(2^{1-s}-1)f(A)\geq0$ and so $f(A)\geq0$.
\end{proof}
In \cite{mos}, Moslehian and Najafi proved the following theorem for positive operators as follows:
\begin{theorem}\label{t4}
Let $A,B\in B(H )^+$. Then $AB + BA$ is positive if and only if
$f(A + B)\leq f(A) + f(B)$ for all non-negative operator monotone functions $f$ on $[0,\infty)$.
\end{theorem}
As an example of operator $s$-convex function, we give the following example.
\begin{example}\label{e3}
Since for every positive operators $A,B\in C(H), ~AB+BA\geq0$, utilizing  Theorem \ref{t4} we get
\begin{equation*}
((1-t)A+tB))^s\leq(1-t)^sA^s+t^sB^s.
\end{equation*}
Therefore the continuous function $ f(t)=t^s ~(0<s\leq 1)$ is operator $s$-convex on $ [0,\infty) $ for operators in $C(H)$.
\end{example}

Dragomir in \cite{dra2} has proved a Hermite-Hadamard type inequality for operator convex function as follows:

\begin{theorem}\label{t5}
Let $f : I \rightarrow \mathbb{R}$ be an operator convex function on the interval $I$. Then
for all selfadjoint operators $A$ and $B$ with spectra in $I$ we have the inequality
\begin{multline*}
\left(f\left(\frac{A+B}{2}\right)\leq\right)\frac{1}{2}\left[f\left(\frac{3A+B}{4}\right)+f\left(\frac{A+3B}{4}\right)\right]\\
\leq\int_0^1 f((1-t)A+tB))dt\\
\leq\frac{1}{2}\left[f\left(\frac{A+B}{2}\right)+\frac{f(A)+f(B)}{2}\right]\left(\leq\frac{f(A)+f(B)}{2}\right).
\end{multline*}
\end{theorem}

Let $X$ be a vector space, $ x, y \in X$, $ x\neq y$. Define the segment
\[[x, y] :=(1- t)x + ty; t \in [0, 1].\]
We consider the function $f : [x, y]\rightarrow \mathbb{R}$ and the associated function
\begin{align*}
&g(x, y) : [0, 1] \rightarrow \mathbb{R},\\
&g(x, y)(t) := f((1 - t)x + ty), t \in [0, 1].
\end{align*}
Note that $f$ is convex on $[x, y]$ if and only if $g(x, y)$ is convex on $[0, 1]$.
For any convex function defined on a segment $[x, y] \in X$, we have the Hermite-Hadamard integral inequality

\begin{equation*}
f\left(\frac{x+y}{2}\right)\leq \int_0^1 f((1-t)x+ty)dt\leq \frac{f(x)+f(y)}{2},
\end{equation*}

which can be derived from the classical Hermite-Hadamard inequality (\ref{1.1}) for the
convex function $g(x, y) : [0,1] \rightarrow \mathbb{R}$.

\begin{lemma}
Let $f:I\subseteq[0,\infty)\rightarrow \Bbb R$ be a continuous function on the interval $I$. Then for every
two positive operators $A,B\in K\subseteq B(H)^+$ with spectra in $I$ the function $f$ is operator s-convex for operators in $[A,B]:=\{(1-t)A+tB : t\in [0,1]\}$
if and only if the function $\varphi_{x,A,B}:[0,1]\rightarrow \Bbb R$ defined by
\begin{equation*}
\varphi_{x,A,B}=\langle f((1-t)A+tB) x,x \rangle
\end{equation*}
is $s$-convex on $[0,1]$ for every $x\in H$ with $\|x\|=1$.
\end{lemma}

\begin{proof}
Let $f$ be operator $s$-convex for operators in $[A,B]$ then for any $t_{1} , t_{2} \in [0,1] $ and $\alpha , \beta \geq 0 $
with $\alpha + \beta = 1$ we have
\begin{multline*}
\varphi_{x,A,B} ( \alpha t_{1} + \beta t_{2} ) = \langle f((1-(\alpha t_{1} +\beta t_{2} ))A +(\alpha t_{1} +\beta t_{2} )B)x,x \rangle\\
=\langle f(\alpha [(1-t_{1})A+t_{1} B]+\beta [(1-t_{2} )A+t_{2} B])x,x \rangle\\
\leq \alpha^s \langle f((1-t_{1} )A+t_{1} B)x,x\rangle + \beta^s \langle f((1-t_{2})A +t_{2} B)x,x\rangle \\
= \alpha^s \varphi_{x,A,B} (t_{1} )+\beta^s \varphi_{x,A,B}(t_2).
\end{multline*}
showing that $\varphi_{x,A,B}$ is a $s$-convex function on $[0, 1]$.

Let now $\varphi_{x,A,B}$ be $s$-convex on $[0,1]$, we show that $f$ is operator $s$-convex for operators in $[A,B]$.
For every  $C=(1-t_{1})A+t_{1}B$ and $D=(1-t_{2})A+t_{2}B$ in $[A,B]$ we have
\begin{multline*}
\langle f(\lambda C+(1-\lambda)D)x,x\rangle\\
=\langle f[\lambda((1-t_{1})A+t_{1}B)+(1-\lambda)((1-t_{2})A+t_{2}B)]x,x\rangle\\
=\langle f[(1-(\lambda t_{1}+(1-\lambda)t_{2}))A+(\lambda t_{1}+(1-\lambda)t_{2})B]x,x\rangle\\
=\varphi_{x,A,B}(\lambda t_{1}+(1-\lambda)t_{2})\\
\leq\lambda^s \varphi_{x,A,B}(t_{1})+(1-\lambda)^s\varphi_{x,A,B}(t_{2})\\
=\lambda^s \langle f((1-t_{1})A+t_{1}B)x,x\rangle+(1-\lambda)^s\langle f((1-t_{2})A+t_{2}B)x,x\rangle\\
\leq\lambda^s \langle f(C)x,x\rangle+(1-\lambda)^s\langle f(D)x,x\rangle.
\end{multline*}
\end{proof}
%------------------------------------------------------------------------------------------------------------------------------
The following theorem is a generalization of Theorem \ref{t1} for operator $s$-convex functions.

\begin{theorem}\label{t6}
Let $f : I \rightarrow \mathbb{R}$ be an operator $s$-convex function on the interval $I\subseteq[0,\infty)$ for operators in $K\subseteq B(H)^+$. Then
for all positive operators $A$ and $B$ in $K$ with spectra in $I$ we have the inequality
\begin{equation}\label{2.1}
2^{s-1}f\left(\frac{A+B}{2}\right)\leq\int_0^1 f((1-t)A+tB)dt\leq\frac{f(A)+f(B)}{s+1}.
\end{equation}
\end{theorem}

\begin{proof}
For $x\in H$ with $\|x\|=1$ and $t\in [0,1]$, we have
\begin{align}\label{2.2}
\left\langle \big((1-t)A+tB\big)x,x\right\rangle = (1-t)\langle Ax,x\rangle+t\langle Bx,x\rangle\in I,
\end{align}
since $\langle Ax,x\rangle\in Sp(A)\subseteq I$ and $\langle Bx,x\rangle\in Sp(B)\subseteq I$.

Continuity of $f$ and (\ref{2.2}) imply that the operator-valued integral $\int_0^1 f((1-t)A+tB)dt$ exists.

Since $f$ is operator $s$-convex, therefore for $t$ in $[0,1]$ and $A,B\in K$ we have
\begin{equation}\label{2.3}
f((1-t)A+tB)\leq (1-t)^sf(A)+t^sf(B).
\end{equation}

Integrating both sides of (\ref{2.3}) over $[0,1]$ we get the following inequality
\begin{equation*}
\int_0^1 f((1-t)A+tB)dt\leq\frac{f(A)+f(B)}{s+1}.
\end{equation*}
To prove the first inequality in (\ref{2.1}) we observe that
\begin{align}\label{2.4}
f\left(\frac{A+B}{2}\right)&\leq \frac{f(tA+(1-t)B)+f((1-t)A+tB)}{2^s}.
\end{align}
Integrating the inequality (\ref{2.4}) over $t\in [0, 1]$ and taking into account
that
\begin{equation*}
\int_0^1f(tA+(1-t)B)dt=\int_0^1f((1-t)A+tB)dt
\end{equation*}
then we deduce the first part of (\ref{2.1}).
\end{proof}

Let $f: I \rightarrow \mathbb{R}$ be operator $s_1$-convex and $g: I \rightarrow \mathbb{R}$ operator $s_2$-convex function on the interval $I$. Then
for all positive operators $A$ and $B$ on a Hilbert space $H$ with spectra in $I$, we define real functions
$M(A,B)$ and $N(A,B)$ on $H$ by
\begin{align*}
M(A,B)(x)&=\langle f(A)x,x\rangle\langle g(A)x,x\rangle+\langle f(B)x,x\rangle\langle g(B)x,x\rangle\quad &&(x\in H),\\
N(A,B)(x)&=\langle f(A)x,x\rangle\langle g(B)x,x\rangle+\langle f(B)x,x\rangle\langle g(A)x,x\rangle\quad &&(x\in H).
\end{align*}

We note that, the Beta and Gamma functions are defined respectively, as
follows:
\[
\beta(x,y)=\int_0^1 t^{x-1}(1-t)^{y-1} dt \quad x>0,~y>0
\]
and
\[
\Gamma(x)=\int_0^\infty e^{-t}t^{x-1}dt\quad x>0.
\]
The following theorem is a generalization of Theorem 3 for operator $s$-convex functions.

\begin{theorem}\label{t7}
Let $f: I \rightarrow \mathbb{R}$ be operator $s_1$-convex and $g: I \rightarrow \mathbb{R}$ operator $s_2$-convex function on the interval $I$ for operators in $K\subseteq B(H)^+$. Then
for all positive operators $A$ and $B$ in $K$ with spectra in $I $, the inequality

\begin{multline}\label{2.5}
\int_{0}^{1}\langle f(tA+(1-t)B)x,x\rangle\langle g(tA+(1-t)B)x,x\rangle dt\\
\leq\frac{1}{s_1+s_2+1}M(A,B)(x)+\beta(s_1+1, s_2+1)N(A,B)(x).
\end{multline}

holds for any $x\in H$ with $\|x\| = 1$.
\end{theorem}

\begin{proof}
For $x\in H$ with $\|x\|=1$ and $t\in [0,1]$, we have
\begin{equation}\label{2.6}
\langle \big(tA+(1-t)B\big)x,x\rangle =t\langle Ax,x\rangle+(1-t)\langle Bx,x\rangle \in I,
\end{equation}
since $\langle Ax,x\rangle\in Sp(A)\subseteq I$ and $\langle Bx,x\rangle\in Sp(B)\subseteq I$.

Continuity of $f,g$ and (\ref{2.6}) imply that the operator valued integrals $\int_0^1 f(tA+(1-t)B)dt$,
 $~\int_0^1 g(tA+(1-t)B)dt$ and $\int_0^1 (fg)(tA+(1-t)B)dt$ exist.

Since $f$ is operator $s_1$-convex and $g$ is operator $s_2$-convex, therefore for $t$ in $[0,1]$ and $x\in H$ we have
\begin{equation}\label{2.7}
\langle f(tA+(1-t)B)x,x\rangle\leq\langle (t^{s_1}f(A)+(1-t)^{s_1}f(B))x,x\rangle,
\end{equation}
\begin{equation}\label{2.8}
\langle g(tA+(1-t)B)x,x\rangle\leq\langle(t^{s_2}g(A)+(1-t)^{s_2}g(B))x,x\rangle.
\end{equation}

From (\ref{2.7}) and (\ref{2.8}) we obtain

\begin{multline}\label{2.9}
\langle f(tA+(1-t)B)x,x\rangle\langle g(tA+(1-t)B)x,x\rangle\\
\leq t^{s_1+s_2}\langle f(A)x,x\rangle\langle g(A)x,x\rangle+(1-t)^{s_1+s_2}\langle f(B)x,x\rangle\langle g(B)x,x\rangle\\
+t^{s_1}(1-t)^{s_2}[\langle f(A)x,x\rangle\langle g(B)x,x\rangle]\\
+t^{s_2}(1-t)^{s_1}[\langle f(B)x,x\rangle \langle g(A)x,x\rangle].
\end{multline}

Integrating both sides of (\ref{2.9}) over $[0,1]$ we get the required inequality (\ref{2.5}).
\end{proof}

The following theorem is a generalization of Theorem 7 in \cite{kirm} for operator $s$-convex functions.

\begin{theorem}\label{t8}
Let $f: I \rightarrow \mathbb{R}$ be operator $s_1$-convex and $g: I \rightarrow \mathbb{R}$ be $s_2$-convex function on the interval $I$ for operators in $K\subseteq B(H)^+$. Then
for all positive operators $A$ and $B$ in $K$ with spectra in $I$, the inequality

\begin{multline}\label{2.10}
2^{s_1+s_2-1}\left\langle f\left(\frac{A+B}{2}\right)x,x\right\rangle\left\langle g\left(\frac{A+B}{2}\right)x,x\right\rangle\\
\leq\int_{0}^{1}\left\langle f(tA+(1-t)B)x,x\rangle\langle g(tA+(1-t)B)x,x\right\rangle dt \\
+\beta(s_1+1,s_2+1)M(A,B)(x)+\frac{1}{s_1+s_2+1}N(A,B)(x),
\end{multline}

holds for any $x\in H$ with $\|x\| = 1$.
\end{theorem}

\begin{proof}
Since $f$ is operator $s_1$-convex and $g$ operator $s_2$-convex, therefore for any $t\in I$ and any $x\in H$ with $\|x\| = 1$ we observe that
\begin{multline*}
\left\langle f\left(\frac{A+B}{2}\right)x,x\right\rangle \left\langle g\left(\frac{A+B}{2}\right)x,x\right\rangle\\
=\left\langle f\left(\frac{tA+(1-t)B}{2}+\frac{(1-t)A+tB}{2}\right)x,x\right\rangle\\
 \times\left\langle g\left(\frac{tA+(1-t)B}{2}+\frac{(1-t)A+tB}{2}\right)x,x\right\rangle
\end{multline*}
\begin{multline*}
\leq\frac{1}{2^{s_1+s_2}}\Big\{\langle f(tA+(1-t)B)x,x\rangle +\langle f(1-t)A+tB)x,x\rangle]\\
\times[\langle g(tA+(1-t)B)x,x\rangle+\langle g((1-t)A+tB)x,x\rangle\Big\}
\end{multline*}
\begin{multline*}
\leq\frac{1}{2^{s_1+s_2}}\Big\{[\langle f(tA+(1-t)B)x,x\rangle \langle g(tA+(1-t)B)x,x\rangle\\
+\langle f((1-t)A+tB)x,x\rangle \langle g((1-t)A+tB)x,x\rangle]\\
+[t^{s_1}\langle f(A)x,x\rangle+(1-t)^{s_1}\langle f(B)x,x\rangle][(1-t)^{s_2}\langle g(A)x,x\rangle+t^{s_2}\langle g(B)x,x\rangle]\\
+[(1-t)^{s_1}\langle f(A)x,x\rangle+t^{s_1}\langle f(B)x,x\rangle][t^{s_2}\langle g(A)x,x\rangle+(1-t)^{s_2}\langle g(B)x,x\rangle]\Big\}
\end{multline*}
\begin{multline*}
=\frac{1}{2^{s_1+s_2}}\Big\{[\langle f(tA+(1-t)B)x,x\rangle \langle g(tA+(1-t)B)x,x\rangle\\
+\langle f((1-t)A+tB)x,x\rangle\langle g((1-t)A+tB)x,x\rangle]\\
+\big(t^{s_1}(1-t)^{s_2}+t^{s_2}(1-t)^{s_1}\big)[\langle f(A)x,x\rangle\langle g(A)x,x\rangle+\langle f(B)x,x\rangle \langle g(B)x,x\rangle]\\
+\big((1-t)^{s_1+s_2}+t^{s_1+s_2}\big)[\langle f(A)x,x\rangle \langle g(B)x,x\rangle+\langle f(B)x,x\rangle \langle g(A)x,x\rangle]\Big\}.
\end{multline*}
By integration over [0,1], we obtain
\begin{multline*}
\left\langle f\left(\frac{A+B}{2}\right)x,x\right\rangle \left\langle g\left(\frac{A+B}{2}\right)x,x\right\rangle\\
\leq\frac{1}{2^{s_1+s_2}}\left(\int_{0}^{1}[\langle f(tA+(1-t)B)x,x\rangle\langle g(tA+(1-t)B)x,x\rangle\right.\\
+\langle f((1-t)A+tB)x,x\rangle\langle g((1-t)A+tB)x,x\rangle]dt\\
\left.+2\beta(s_1+1,s_2+1)M(A,B)(x)+\frac{2}{s_1+s_2+1}N(A,B)(x)\right).
\end{multline*}

This implies the required inequality (\ref{2.10}).
\end{proof}


\begin{thebibliography}{99}

\bibitem{bak} M. Klari$\breve{c}$i$\acute{c}$ Bakula and J. Pe$\breve{c}$ari$\acute{c}$
\textit{Note on some Hadamard-type inequalities},  J. Inequal.
 Pure Appl. Math. \textbf{5} (2004), no. 3, Article 74.

\bibitem{bar} N.S. Barnett, P. Cerone and S.S. Dragomir, \textit{Some new inequalities for Hermite-Hadamard
divergence in information theory}, Stochastic analysis and applications. Vol. 3, 7-19, Nova Sci.
Publ., Hauppauge, NY, 2003.

\bibitem{dra1}
S.S. Dragomir and S. Fitzpatrick, \textit {The Hadamard's inequality for s-convex functions in the second sense}, Demonstratio Math.,\textbf{32}(4)(1999), 687-696.

\bibitem{dra2}
S.S. Dragomir,\textit {The Hermite-Hadamard type inequalities for operator convex functions}, Appl. Math. Comput. \textbf{218}(3)(2011), 766-772.
%********************************************************************************************************************************************************
\bibitem{dra3} S.S. Dragomir and C.E.M. Pearce, \textit{Selected Topics on Hermite-Hadamard Inequalities and applications},
(RGMIA Monographs http:// rgmia.vu.edu.au/ monographs/ hermite hadamard.html), Victoria
University, 2000.
\bibitem{fur}
T. Furuta, J. Mi$\acute{c}$i$\acute{c}$ Hot, J. Pe$\breve{c}$ari$\acute{c}$ and Y. Seo, \textit {Mond-Pe$\breve{c}$ari$\acute{c}$
 Method in Operator Inequalities. Inequalities for Bounded Selfadjoint Operators on a Hilbert Space}, Element, Zagreb,
2005.

\bibitem{hud} H. Hudzik and L. Maligranda, \textit {Some remarks on s-convex functions}, Aequationes Math., \textbf{48}
(1994), 100-111.

\bibitem{kik} E. Kikianty, \textit{Hermite-Hadamard inequality in the geometry of Banach spaces}, PhD thesis,
Victoria University, 2010.

\bibitem{kirm} U.S. Kirmaci, M.K. Bakula, M.E. \"{O}zdemir and J. Pe\v{c}ari\'c, \textit {Hadamard-type inequalities for
s-convex functions}, Appl. Math. and Compt., \textbf{193} (2007), 26-35.

\bibitem{mos} M. S. Moslehian and H. Najafi, \textit {Around operator monotone functions}, Integr. Equ. Oper. Theory \textbf{71} (2011), 575-582.

\bibitem{pach} B. G. Pachpatte, \textit {On some inequalities for convex functions}, RGMIA Res.
Rep. Coll., \textbf{6} (E), (2003).

\bibitem{tun}
M.Tun\c{c}, \textit {On some new inequalities for convex functions}, Turk. J. Math. \textbf{36} (2012), 245-251.

\bibitem{wu} S.Wu, \textit{On the weighted generalization of the Hermite-Hadamard inequality and its applications},
Rocky Mountain J. Math. \textbf{39} (2009), no. 5, 1741-1749.

%*********************************************************************************************************************************************************
\end{thebibliography}
\end{document}